\documentclass{IEEEtran}

\usepackage{textcomp}

\usepackage{graphicx}
\usepackage{algorithmic}
\usepackage{amsmath,amssymb,amsfonts}
\usepackage{siunitx}
\usepackage{longtable,tabularx}
\usepackage{xcolor}
\usepackage{cases}
\usepackage{mathtools}
\usepackage{siunitx}
\usepackage{placeins}
\usepackage{amsfonts}
\usepackage{hyperref}
\usepackage{bm}
\usepackage{cite}
\usepackage{pifont}
\usepackage{bm}
\usepackage{blindtext}
\usepackage{nicematrix}
\usepackage{subcaption}
\usepackage[normalem]{ulem}

\usepackage[most]{tcolorbox}

\usepackage{tikz}
\usetikzlibrary{fadings}
\usetikzlibrary{patterns}
\usetikzlibrary{shadows.blur}
\usetikzlibrary{shapes}

\usepackage{amsthm}

\newtheorem{theorem}{Theorem}[section]
\newtheorem{proposition}{Proposition}[section]

\newtheorem{remark}{Remark}[section]
\newtheorem{assumption}{Assumption}[section]

\newcommand{\I}{ ~\forall i\in \mathbb{N}_I}

\newcommand{\K}{~\forall k\in \mathbb{N}_{T+1}}
\newcommand{\Km}{~\forall k\in \mathbb{N}_{T}}

\usepackage[capitalise]{acro}
\DeclareAcronym{RMPC}{
  short = RMPC,
  long  = robust model predictive control,
}
\DeclareAcronym{SLS}{
  short = SLS,
  long  = system level synthesis,
}
\DeclareAcronym{MPC}{
  short = MPC,
  long  = model predictive control,
}

\DeclareAcronym{LTV}{
  short = LTV,
  long  = linear time-varying,
}
\DeclareAcronym{NLP}{
  short = NLP,
  long  = nonlinear program,
}

\DeclareAcronym{SQP}{
  short = SQP,
  long  = sequential quadratic programming,
}

\DeclareAcronym{QP}{
  short = QP,
  long  = quadratic program,
}

\newcommand{\ek}{e_k}
\newcommand{\sv}[2]{\left(#1,#2\right)}

\newcommand{\zv}{\sv{z_k}{v_k}}
\newcommand{\T}{^\top}

\newcommand{\B}{ \mathbb{B}_\infty}
\newcommand{\xz}{\bar{{x}}}
\newcommand{\Z}{\mathbf{z}}
\newcommand{\V}{\mathbf{v}}
\newcommand{\U}{\mathbf{u}}
\newcommand{\Y}{\mathbf{y}}
\newcommand{\X}{\mathbf{x}}
\newcommand{\W}{\mathbf{w}}
\renewcommand{\d}{d}
\newcommand{\E}{\bm{\d}}
\newcommand{\tube}{\tau}
\newcommand{\D}{\bm{\tube}}
\newcommand{\defmath}{\vcentcolon =}
\newcommand{\mathdef}{=\vcentcolon}
\newcommand{\C}{\mathbb{P}}
\newcommand{\feas}{\diamond}

\title{Robust Nonlinear Optimal Control\\ via System Level Synthesis}

\author{Antoine P. Leeman$^1$, Johannes Köhler$^1$, Andrea Zanelli$^1$, Samir Bennani$^2$,  Melanie N. Zeilinger$^1$
\thanks{This work has been supported by the European Space Agency under OSIP 4000133352, the Swiss Space Center, and the Swiss National Science Foundation under NCCR Automation (grant agreement 51NF40 180545).}
\thanks{$^1$Antoine P. Leeman, Johannes Köhler, Andrea Zanelli, and  Melanie N. Zeilinger are with the Institute for Dynamic Systems and Control,
ETH Zürich, Zürich 8053, Switzerland (email: aleeman@ethz.ch; jkoehle@ethz.ch; zanellia@ethz.ch; mzeilinger@ethz.ch)}
\thanks{$^2$Samir Bennani is with the European Space Agency, Noordwijk 2201AZ, The Netherlands (email: samir.bennani@esa.int)}
}

\begin{document}
\maketitle

\begin{abstract}
This paper addresses the problem of finite horizon constrained robust optimal control for nonlinear systems subject to norm-bounded disturbances.
To this end, the underlying uncertain nonlinear system is decomposed based on a first-order Taylor series expansion into a nominal system and an error (deviation) described as an uncertain linear time-varying system. 
This decomposition allows us to leverage system level synthesis to jointly optimize an affine error feedback, a nominal nonlinear trajectory, and, most importantly, a dynamic linearization error over-bound used to ensure robust constraint satisfaction for the nonlinear system.
The proposed approach thereby results in less conservative planning compared with state-of-the-art techniques.
We demonstrate the benefits of the proposed approach to control the rotational motion of a rigid body subject to state and input constraints.

\end{abstract}

\section{Introduction}
Robust nonlinear optimal control represents one of the central problems, arising, e.g., in trajectory optimization or \ac{MPC}, in many safety-critical applications, involving, e.g., robotic systems, drones, spacecraft, and many others.
While this problem has been extensively studied in the literature~\cite{Gruene2017}, and rigorous constraint satisfaction properties can be derived in the presence of disturbances (see robust predictive control formulations, e.g.,~\cite{Yu2013TubeSystems, Singh2017RobustOptimization, Kohler2021ASystems, Rakovic2023}), this is commonly achieved at the cost of introducing conservativeness.

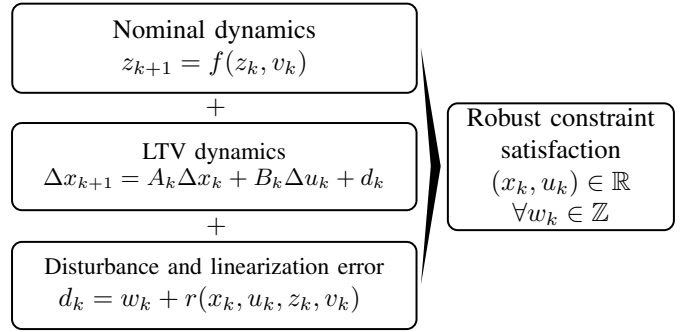
\begin{figure}[t!]
    \centering
\tikzset{every picture/.style={line width=0.75pt}} 

\begin{tikzpicture}[x=0.75pt,y=0.75pt,yscale=-1,xscale=1]

\draw  [draw opacity=0][fill={rgb, 255:red, 0; green, 0; blue, 0 }  ,fill opacity=1 ] (234.2,39.8) -- (244.2,99.8) -- (234.2,159.8) -- (239.2,99.8) -- cycle ;

\draw    (28,83) .. controls (28,80.24) and (30.24,78) .. (33,78) -- (227,78) .. controls (229.76,78) and (232,80.24) .. (232,83) -- (232,117) .. controls (232,119.76) and (229.76,122) .. (227,122) -- (33,122) .. controls (30.24,122) and (28,119.76) .. (28,117) -- cycle  ;
\draw (130,100) node  [font=\small] [align=left] {\begin{minipage}[lt]{136pt}\setlength\topsep{0pt}
\begin{center}
LTV dynamics \\$\displaystyle \Delta x_{k+1} =A_{k} \Delta x_{k} +B_{k} \Delta u_{k} +d_{k}$
\end{center}

\end{minipage}};
\draw    (28.12,22.75) .. controls (28.12,19.99) and (30.36,17.75) .. (33.12,17.75) -- (227.12,17.75) .. controls (229.88,17.75) and (232.12,19.99) .. (232.12,22.75) -- (232.12,57.75) .. controls (232.12,60.51) and (229.88,62.75) .. (227.12,62.75) -- (33.12,62.75) .. controls (30.36,62.75) and (28.12,60.51) .. (28.12,57.75) -- cycle  ;
\draw (130.12,40.25) node   [align=left] {\begin{minipage}[lt]{136.17pt}\setlength\topsep{0pt}
\begin{center}
Nominal dynamics\\$\displaystyle z_{k+1} =f( z_{k} ,v_{k})$\\
\end{center}

\end{minipage}};
\draw    (28,143) .. controls (28,140.24) and (30.24,138) .. (33,138) -- (227,138) .. controls (229.76,138) and (232,140.24) .. (232,143) -- (232,177) .. controls (232,179.76) and (229.76,182) .. (227,182) -- (33,182) .. controls (30.24,182) and (28,179.76) .. (28,177) -- cycle  ;
\draw (130,160) node   [align=left] {\begin{minipage}[lt]{136pt}\setlength\topsep{0pt}
\begin{center}
{\small Disturbance and linearization error}\\$\displaystyle d_{k} =w_{k} +r( x_{k} ,u_{k} ,z_{k} ,v_{k}) \ $
\end{center}

\end{minipage}};
\draw    (247.81,73.07) .. controls (247.81,70.31) and (250.05,68.07) .. (252.81,68.07) -- (356.81,68.07) .. controls (359.57,68.07) and (361.81,70.31) .. (361.81,73.07) -- (361.81,127.07) .. controls (361.81,129.83) and (359.57,132.07) .. (356.81,132.07) -- (252.81,132.07) .. controls (250.05,132.07) and (247.81,129.83) .. (247.81,127.07) -- cycle  ;
\draw (304.81,100.07) node   [align=left] {\begin{minipage}[lt]{75.06pt}\setlength\topsep{0pt}
\begin{center}
Robust constraint satisfaction\\$\displaystyle ( x_{k} ,u_{k}) \in \mathbb{P}$\\$\displaystyle \forall w_{k} \in \mathbb{W}$
\end{center}

\end{minipage}};
\draw (124,64.07) node [anchor=north west][inner sep=0.75pt]    {$+$};
\draw (124.33,124.07) node [anchor=north west][inner sep=0.75pt]    {$+$};

\end{tikzpicture}
    \caption{Decomposition of the uncertain nonlinear dynamics into nominal nonlinear dynamics, \ac{LTV} dynamics (Section~\ref{sec:SLS}), and an error term composed of the disturbance and dynamic linearization error (Section~\ref{sec:linerror}), facilitating linear error feedback optimization with robust constraint satisfaction for the constrained nonlinear uncertain system (Section~\ref{sec:result}).}
    \label{fig:outline}
\end{figure}

The robust control design task is traditionally divided into two main steps: the nominal trajectory optimization~\cite{MALYUTA2021282} and the offline design of a stabilizing feedback~\cite{Singh2017RobustOptimization} compensating for modeling errors or disturbances.
To ensure robust satisfaction of safety-critical state constraints, the nominal trajectory optimization is coupled with the over-approximation of the error reachable set using, e.g., tubes or funnels~\cite{Majumdar2017FunnelPlanning}.
There exists a wide range of techniques to construct corresponding over-approximations of the tubes/funnels (cf., e.g.,~\cite{Althoff2008ReachabilityLinearization, Cannon2011RobustControl, Houska2011RobustSystems, Majumdar2017FunnelPlanning, Yu2013TubeSystems, Singh2017RobustOptimization, Kohler2021ASystems,Rakovic2023}), however,
these methods can introduce significant conservatism, especially due to the choice of an offline fixed error feedback.

The conservativeness of an offline-determined error feedback policy can be addressed for linear systems by directly predicting robust control invariant polytopes~\cite{LANGSON2004125, VILLANUEVA2024111543}. Compare also~\cite{Villanueva2017RobustInequalities} for a recent approach for nonlinear systems.
Alternatively, a min-max formulation that combinatorially considers all disturbance permutations can reduce conservativeness in the linear case~\cite{Scokaert1998}.
Another systematic approach to jointly optimize a linear feedback while considering constraints is presented in~\cite{Messerer2021AnFeedback, Kim2022JointNonlinearities}, which extends approximate ellipsoidal disturbance propagation~\cite{Houska2011RobustSystems, Zanelli2021Zero-orderSets} to include optimized feedback policies.
Other methods to obtain feedback policies and (optimal) trajectories, which come without principled guarantees for robust constraint satisfaction (cf.~\cite{Neunert2016FastTracking, Gramlich2022, Li2011IterativeSystems}), are used in practice.

However, all the above mentioned methods that jointly optimize over nominal trajectories and feedback policies result in an over- or under-approximation of the true reachable sets, even for linear systems. We overcome this limitation by extending \ac{SLS}~\cite{Anderson2019, Sieber2021AControl, Sieber2021SystemMPC,Chen2021SystemApproximation, wang2019system}, or equivalently affine disturbance feedback~\cite{Goulart2006OptimizationConstraints, Goulart2006AffineConstraints, Skaf2010DesignOptimization} to constrained nonlinear systems. In particular, for linear systems, SLS allows to jointly optimize a linear error feedback policy and nominal trajectory and thereby provide an exact reachable set~\cite{Sieber2021AControl}.
There exist nonlinear \ac{SLS} extensions~\cite{Furieri2022NeuralSystems, Conger2022NonlinearSystems, Ho2020}, which however do not consider (robust) constraint satisfaction. 
\subsubsection*{Contribution}
We propose a novel approach for optimal control of nonlinear systems with robust constraint satisfaction. 
As shown in Fig.~\ref{fig:outline}, the nonlinear system is expressed as a \ac{LTV} error system constructed around a jointly optimized nominal trajectory. This formulation includes a jointly optimized error term corresponding to the linearization error. 

The presented method has the following advantages compared to the literature:
\begin{itemize}
    \item The nominal nonlinear trajectory, the affine error feedback and the dynamic linearization error over-bounds are jointly optimized for robust constraint satisfaction, leading to reduced conservativeness. This formulation leverages and extends \ac{SLS} to the nonlinear constrained case.
    \item The reachable set used for robust constraint satisfaction is tight for linear systems with affine feedback.
    \item The resulting nonlinear optimization problem does \textit{not} require linear matrix inequalities as in~\cite{Kim2022JointNonlinearities,Villanueva2017RobustInequalities}, or system-specific offline design as in~\cite{Rakovic2023,Kohler2021ASystems}, and allows for an efficient implementation (e.g., using \ac{SQP}), thus ensuring scalability.
\end{itemize}
We demonstrate the conservativeness reduction by comparing with nominal trajectory optimization, open-loop robust trajectory optimization, i.e., without controller optimization, linear techniques~\cite{Sieber2021AControl,Chen2021SystemApproximation}, i.e., with linear nominal trajectory and offline overbounding, as well as online overbounding of nonlinearity using multiplicative uncertainty (Section~\ref{sec:case_study}).

\subsubsection*{Notation}
We define the set $\mathbb{N}_T \defmath \{0,\dots,T-1\}$ where $T$ is a natural number. We denote stacked vectors or matrices by $\sv{a}{b} = [a\T~b\T]\T$. For a vector $r\in \mathbb{R}^n$, we denote its $i^\text{th}$ component by $r_i$. Let $\mathbb{R}$ be the set of real numbers, $0_{p,q}\in \mathbb{R}^{p\times q}$ be a matrix of zeros, and $0_m\in \mathbb{R}^m$ be a vector of zeros. Let $\mathcal{L}^{T,p\times q}$ denote the set of all block lower-triangular matrices with the following structure
\begin{equation}
     \mathcal{M} = \begin{bmatrix} M^{0,0} & 0_{p,q} & \dots & 0_{p,q} \\ M^{1,1} & M^{1,0} & \dots & 0_{p,q} \\ \vdots & \vdots & \ddots & \vdots \\M^{T-1,T-1} & M^{T-1,T-2} & \dots & M^{T-1,0} \end{bmatrix},
     \label{eq:mat}
\end{equation}
where $M^{i,j}\in \mathbb{R}^{p\times q}$.
The block diagonal matrix consisting of matrices $A_1,\dots,A_T$ is denoted by $\mathrm{blkdiag}(A_1,\dots,A_T)$.
The matrix $\mathcal{I}$ denotes the identity with its dimensions either inferred from the context or indicated by the subscript, i.e., $\mathcal{I}_{n_\mathrm{x}}\in\mathbb{R}^{n_\mathrm{x}\times n_\mathrm{x}}$.
Let $\B^{m}$, be the unit ball defined by $\B^{m} \defmath \{d\in\mathbb{R}^m|~ \|d\|_\infty \le 1\}$.
For a matrix $M\in \mathbb{R}^{n\times m}$, the $\infty$-norm is given by $\|M \|_\infty \defmath  \max_{d\in \B^{m}} \|Md\|_\infty$. 
For two sets $\mathbb{W}_1, \mathbb{W}_2\subseteq \mathbb{R}^n$, the Minkowski sum is defined as $\mathbb{W}_1 \oplus \mathbb{W}_2 \defmath  \{ w_1 + w_2 |~w_1 \in \mathbb{W}_1, w_2 \in \mathbb{W}_2\}$. We define $\mathbb{W}^k\defmath \underbrace{\mathbb{W} \times \dots \times \mathbb{W}}_{\text{$k$ times}}$.
For a sequence of vectors $w_k\in\mathbb{W}_k\subseteq \mathbb{R}^m$ and $k\in\mathbb{N}$, we define $\W_{0:k}:=(w_0,\dots,w_k)\in \mathbb{W}^{0:k}\defmath\mathbb{W}_0 \times \dots \times \mathbb{W}_{k}$. 

\section{Problem Formulation}
\label{sec:problem_formulation}
We consider the following robust nonlinear optimal control problem: \begin{IEEEeqnarray}{rCl}
\IEEEyesnumber\label{eq:prob_form}
\IEEEyessubnumber*
\min_{\pi(\cdot)} \quad && J_T(\xz, \pi(\cdot)),\label{eq:cost_horizon}\\
    \text{s.t.}\quad  && x_{k+1}= f(x_k,u_k) + w_k\Km,\label{eq:nonlinear_dyn}\\
    && x_0 = \xz,\label{eq:initial_conditions_prob}\\
    &&u_k = \pi_k(\X_{0:k})\Km,\label{eq:robust_feedback}\\
    &&(x_k,u_k) \in\C\K~\forall w_k \in \mathbb{W}.\label{eq:prob_form_cons}
\end{IEEEeqnarray}
The dynamics are given by~\eqref{eq:nonlinear_dyn}, with the state $x_k\in \mathbb{R}^{n_\mathrm{x}}$, the input $u_k\in \mathbb{R}^{n_\mathrm{u}}$ and the  disturbance $w_k\in \mathbb{W}\subseteq \mathbb{R}^{n_\mathrm{x}}$ at time step $k$. The initial condition is given by $\xz\in \mathbb{R}^{n_\mathrm{x}}$ in~\eqref{eq:initial_conditions_prob}.
The control input is obtained by optimizing over general causal policies $\pi_k$~\eqref{eq:robust_feedback}, with $\pi = ( \pi_0, \dots, \pi_{T}): \mathbb{R}^{(T+1)n_\mathrm{x}}\mapsto\mathbb{R}^{(T+1)n_\mathrm{u}}$, and the last input $u_T$ is kept in the problem formulation for notational convenience. We primarily focus on the robust constraint satisfaction~\eqref{eq:prob_form_cons} and, for simplicity, consider a general cost~\eqref{eq:cost_horizon} over the prediction horizon $T$ which does not depend on $w$.
\begin{assumption}
The constraint set $\C$~\eqref{eq:prob_form_cons} is a compact polytopic set with $\C \defmath  \{(x,u)|~ c_i\T(x,u)+b_i\leq 0, \forall i\in \mathbb{N}_I\}$, where $c_i \in \mathbb{R}^{n_\mathrm{x}+n_\mathrm{u}}$ and $b_i \in \mathbb{R}$.
\label{assum:2}
\end{assumption}

Problem~\eqref{eq:prob_form} is not computationally tractable because of the optimization over the general feedback policy $\pi(\cdot)$ and the robust constraint satisfaction required in~\eqref{eq:prob_form_cons}.
Consequently, the goal is to find a feasible, but potentially sub-optimal, solution to this problem, i.e., a feedback policy that ensures robust state and input constraint satisfaction.
 To this end, we define a nominal trajectory as 
 \begin{equation}
     z_{k+1}= f(z_k,v_k)\Km,~ z_0 = \xz, \label{eq:nominal_trajectory}
\end{equation}
and restrict the policy to a causal affine time-varying error feedback 
\begin{equation}
\label{eq:affine_feedback}
    \pi_k(\X_{0:k}) = v_k + \sum_{j=0}^{k-1} K^{k-1,j} \Delta x_{k-j}
\end{equation}
with $\pi_k(\X_{0:k}) : \mathbb{R}^{(k+1) n_\mathrm{x}}\mapsto \mathbb{R}^{n_\mathrm{u}}$, $v_k\in \mathbb{R}^{n_\mathrm{u}}$, $z_k\in \mathbb{R}^{n_\mathrm{x}}$, $K^{i,j}\in \mathbb{R}^{n_\mathrm{u}\times n_\mathrm{x}}$, and the errors $\Delta x_k  \defmath  x_k  - z_k $, $\Delta u_k  \defmath u_k - v_k$.
We consider the following standard assumptions.
\begin{assumption}
The nonlinear dynamics~\eqref{eq:nonlinear_dyn} $f: \mathbb{R}^{n_\mathrm{x}} \times \mathbb{R}^{n_\mathrm{u}} \mapsto \mathbb{R}^{n_\mathrm{x}}$ are twice continuously differentiable.
\label{assum:1}
\end{assumption}
Assumption~\ref{assum:1} implies local Lipschitz continuity, which enables bounding of the linearization error in Proposition~\ref{prop:bound_remainder}.
\begin{assumption}
The disturbance $w_k$ belongs to a bounded set $\mathbb{W}$ defined as
\begin{equation}
w_k \in \mathbb{W} \defmath \{Ed|~d\in \B^{n_\mathrm{w}}\} = E\B^{n_\mathrm{w}} \subseteq \mathbb{R}^{n_\mathrm{x}}\Km,
\label{eq:noise_set}
\end{equation}
with $E\in \mathbb{R}^{n_\mathrm{x}\times n_\mathrm{w}}$.
\label{assum:3}
\end{assumption}
\section{Robust Nonlinear Optimal Control via \ac{SLS}}
In this section, we derive the main result of the paper using the steps depicted in Fig.~\ref{fig:outline}, i.e., we propose a formulation to optimize over affine policies~\eqref{eq:affine_feedback} that provide robust constraint satisfaction for the nonlinear system~\eqref{eq:nonlinear_dyn}. 
We decompose the nonlinear system equivalently as the sum of a nominal nonlinear system and an \ac{LTV} error system that accounts both for the local linearization error (Section~\ref{sec:linerror}) and the additive disturbance. Using established \ac{SLS} tools for \ac{LTV} systems (Section~\ref{sec:SLS}), we parameterize the closed-loop response for this \ac{LTV} system. As a result, we obtain an optimization problem that jointly optimizes the nominal nonlinear trajectory~\eqref{eq:nominal_trajectory}, the error feedback~\eqref{eq:affine_feedback}, and the dynamic linearization error over-bounds, all while guaranteeing robust constraint satisfaction (Section~\ref{sec:result}). 
\subsection{Over-approximation of nonlinear reachable set}
\label{sec:linerror}
The goal of this section is to decompose the uncertain nonlinear system into a nominal nonlinear system and an \ac{LTV} error system subject to some disturbance. The linearization of the dynamics~\eqref{eq:nominal_trajectory} around a nominal state and input $(z,v)$ is characterized by the Jacobian matrices:
\begin{IEEEeqnarray}{rCl}
&A(z, v) \defmath  \left.\frac{\partial  f}{\partial x}   \right|_{(x,u) = (z,v)},~B(z, v) \defmath \left.\frac{\partial  f}{\partial u}   \right|_{(x,u) = (z,v)}.
\label{eq:jac}
\end{IEEEeqnarray}
Using the Lagrange form of the remainder of the Taylor series expansion, we obtain 
\begin{IEEEeqnarray}{rLl}
    f(x ,u ) + w &\stackrel{\eqref{eq:nominal_trajectory}}{=} f(z , v )
    +  A(z ,v )( x  - z ) +B(z ,v )( u  - v )  \nonumber\\
   & ~+ \underbrace{r(x ,u ,z ,v )+ w}_{\mathdef d}, \label{eq:lin_hess}
\end{IEEEeqnarray}
with the remainder $r: \mathbb{R}^{n_\mathrm{x}}\times\mathbb{R}^{n_\mathrm{u}}\times\mathbb{R}^{n_\mathrm{x}}\times\mathbb{R}^{n_\mathrm{u}} \mapsto \mathbb{R}^{n_\mathrm{x}}$, where both the disturbance $w$ and the remainder $r(x ,u ,z ,v )$ are lumped in the disturbance $\d\defmath r(x ,u ,z ,v )+ w \in\mathbb{R}^{n_\mathrm{x}}$. Hence, the overall uncertain prediction is decomposed into a nominal dynamics, LTV error dynamics, and a remainder term, as summarized in Fig.~\ref{fig:outline}.
 
To bound the remainder, we consider the (symmetric) Hessian $ H_i:  \mathbb{R}^{n_\mathrm{x}+n_\mathrm{u}}\mapsto\mathbb{R}^{(n_\mathrm{x}+n_\mathrm{u})\times(n_\mathrm{x} + n_\mathrm{u})}$  of the $i^\text{th}$ component of $f$, i.e.,
\begin{equation}
 H_i(\xi) = \left.\begin{bmatrix}
     \frac{\partial^2  f_i}{\partial x^2}  & \frac{\partial^2  f_i}{\partial x\partial u}\\
     * & \frac{\partial^2  f_i}{\partial u^2} 
 \end{bmatrix}\right|_{\sv{x}{u} = \xi},
\end{equation} 
where $\xi\in \mathbb{R}^{n_\mathrm{x}+n_\mathrm{u}}$ lies between the linearization point $(z,v)$ and the evaluation point $(x,u)$. We define the constant\footnote{Note that $\max_{\|h\|_\infty\le 1} |h\T Hh| \le \sum_i\sum_j |H_{ij}| $, with $H_{ij}$, the element on the $i^\text{th}$ row and $j^\text{th}$ column of the matrix $H$.} $\mu_i\in\mathbb{R}^{n_\mathrm{x} \times n_\mathrm{x} }$ as element-wise (worst-case) curvature, i.e., Hessian over-bounded in the constraint set $\mathbb{P}$:
\begin{equation}
    \mu \defmath \text{diag}(\mu_1, \dots, \mu_{n_\mathrm{x}}),~\mu_i \defmath  \frac{1}{2} \max_{\xi \in\C, \|h\|_\infty\le 1} |h\T H_i(\xi)h|,
    \label{eq:def_mu}
\end{equation}
and the error $e_k  \defmath  \sv{\Delta x_k }{\Delta u_k }\in \mathbb{R}^{n_\mathrm{x}+n_\mathrm{u}}$.
\begin{proposition}
\label{prop:bound_remainder}
Given Assumptions~\ref{assum:1} and~\ref{assum:2}, the remainder in~\eqref{eq:lin_hess} satisfies
\begin{equation}
    |r_i(x ,u ,z ,v )| \le  \|e \|_\infty^2\mu_i,
    \label{eq:rem_def}
\end{equation}
for any $(x ,u )\in\C, (z ,v )\in\C$.
\end{proposition}
\begin{proof}
Applying the definition of the Lagrange form of the remainder, for all $(x,u,z,v)$ and for all $i$, there exists a $\xi\in\C$ (using convexity) such that
\begin{align}
    |r_i(x ,u ,z ,v )| 
    & = \frac{1}{2}| e \T H_i(\xi) e |\nonumber \\
    &\le  \max_{\xi \in\C} \frac{1}{2} |e \T H_i(\xi) e |\stackrel{\eqref{eq:def_mu}}{\leq}\| e \|_\infty^2 \mu_i.&&\qedhere
\nonumber
\end{align}
\end{proof}
\begin{remark}
 The evaluation of the constants $\mu_i$ can be achieved using, e.g., sampling, interval arithmetics \cite{limon2005robust} or tailored nonlinearity bounds as in \cite{Houska2011RobustSystems}. 
\end{remark}
The bound on the Lagrange form of the remainder used in~\eqref{eq:rem_def} has two factors: the (static) offline-fixed constants $\mu_i$ --the only offline design required for the proposed method-- and the (dynamic) jointly-optimized error $\|e\|_\infty$.
Crucially, we later co-optimize a dynamic over-bound on the Lagrange remainder~\eqref{eq:rem_def} with bounds on the error $e$.
Due to Proposition~\ref{prop:bound_remainder} and Assumption~\ref{assum:3}, the combined disturbance
$\d $ from~\eqref{eq:lin_hess} satisfies
\begin{equation}
\begin{aligned}
    \d \defmath w  + r(x ,u ,z ,v )&\in E\B^{n_\mathrm{w}} \oplus \|e \|_\infty^2\mu \B^{n_\mathrm{x}}\\
    & =  [E,\|e\|_\infty^2 \mu]\B^{n_\mathrm{w}+n_\mathrm{x}}.
\end{aligned}
    \label{eq:eta_kdef}
\end{equation}
Using this expression, we can compute an outer approximation of the reachable set of the nonlinear system, using the following \ac{LTV} error system
\begin{equation}
\Delta x_{k+1} = A_k \Delta x_k + B_k \Delta u_k + \d_k \K, ~\Delta x_0 =0_{n_\mathrm{x}},
\label{eq:LTV_based_system}
\end{equation}
where $A_k \defmath A(z_k, v_k)$, $B_k \defmath B(z_k, v_k)$ are the Jacobians of the dynamics defined as in~\eqref{eq:jac}.

Similar \ac{LTV} error dynamics are used in~\cite{Althoff2008ReachabilityLinearization,Cannon2011RobustControl,Villanueva2017RobustInequalities,Houska2011RobustSystems, Kim2022JointNonlinearities} to over-approximate the reachable set. To optimize affine error feedback~\eqref{eq:affine_feedback} while ensuring robust constraint satisfaction of the nonlinear system based on this \ac{LTV} error dynamics, we next consider the robust optimal control problem for the special case of \ac{LTV} error systems as in~\eqref{eq:LTV_based_system}.
\subsection{LTV error dynamics}
\label{sec:SLS}
In this section, we consider the parametrization of linear error feedbacks for \ac{LTV} systems using established \ac{SLS} techniques~\cite{Anderson2019}, which provides the basis for addressing the nonlinear case via~\eqref{eq:lin_hess}.
In order to ultimately address the error dynamics~\eqref{eq:LTV_based_system}, we consider a given nominal trajectory $z_k,v_k$, denote the error $\tilde{x}=x-z$, $\tilde{u}=u-v$, and assume the following error dynamics analogous to~\eqref{eq:LTV_based_system} for ease of exposition:
\begin{equation}
     \tilde x_{k+1} = \tilde A_k \tilde x_k + \tilde B_k  \tilde u_k + \tilde w_k\Km,~\tilde x_0 =0_{n_\mathrm{x}},
    \label{eq:sys_dyn}
\end{equation}
with $\tilde A_k\in \mathbb{R}^{n_\mathrm{x}\times n_\mathrm{x}}$, $\tilde B_k\in \mathbb{R}^{n_\mathrm{x}\times n_\mathrm{u}}$, and 
\begin{equation}
    \tilde{w}_k\in\tilde{\mathbb{W}}_k\defmath\tilde{E}_k\B^{n_{\tilde{\mathrm{w}}}},\label{eq:noise_set_lin}
\end{equation}
with some $\tilde E_k\in \mathbb{R}^{n_\mathrm{x}\times n_{\tilde{\mathrm{w}}}}$.
The notation $\tilde{\cdot}$ differentiates~\eqref{eq:sys_dyn} from the Taylor expansion of the nonlinear system~\eqref{eq:LTV_based_system}.
The dynamics~\eqref{eq:sys_dyn} are written compactly as
\begin{equation}
\tilde \X = \mathcal{Z}\tilde{\mathcal{A}}\tilde \X + \mathcal{Z}\tilde{\mathcal{B}} \tilde \U+ \tilde\W,
\label{eq:LTV}
\end{equation}
with  $\tilde\W \defmath  \left(\tilde  w_0, \dots,\tilde w_{T-1}\right)\in \mathbb{R}^{Tn_\mathrm{x}}$, $\tilde \X \defmath  \left(\tilde x_1, \dots,\tilde x_T\right)\in \mathbb{R}^{Tn_\mathrm{x}}$, $\tilde \U \defmath  \left( \tilde u_1,\dots,\tilde u_T \right)\in \mathbb{R}^{Tn_\mathrm{u}}$, $ \tilde{\mathcal{A}} \defmath  \text{blkdiag}(\tilde A_1, \dots, \tilde A_{T-1},0_{n_\mathrm{x},n_\mathrm{x}})\in\mathcal{L}^{T,n_\mathrm{x}\times n_\mathrm{x}}$, $ \tilde{\mathcal{B}}\defmath  \text{blkdiag}(\tilde B_1, \dots, \tilde B_{T-1},0_{n_\mathrm{x},n_\mathrm{u}})\in\mathcal{L}^{T,n_\mathrm{x}\times n_\mathrm{u}}$ and the block-lower shift matrix $\mathcal{Z}\in\mathcal{L}^{T,n_\mathrm{x} \times n_\mathrm{x}}$ is given by
\begin{equation}
\mathcal{Z}\defmath
\begin{bmatrix} 
0_{n_\mathrm{x},n_\mathrm{x}} & 0_{n_\mathrm{x},n_\mathrm{x}}   & \dots     & 0_{n_\mathrm{x},n_\mathrm{x}} \\
\mathcal{I}_{n_\mathrm{x}} & 0_{n_\mathrm{x},n_\mathrm{x}}   & \dots     & 0_{n_\mathrm{x},n_\mathrm{x}} \\ 
\vdots      & \ddots        & \ddots    & \vdots \\
0_{n_\mathrm{x},n_\mathrm{x}} & \dots   & \mathcal{I}_{n_\mathrm{x}}    & 0_{n_\mathrm{x},n_\mathrm{x}} 
\end{bmatrix}.
\end{equation}
We introduce the causal linear feedback $\tilde \U = \tilde{\mathcal{K}} \tilde\X$, $\tilde{\mathcal{K}}\in \mathcal{L}^{T,n_\mathrm{u}\times n_\mathrm{x}}$ i.e., $\tilde u_k =\sum_{j=0}^{k-1} \tilde K^{k-1,j}\tilde x_{k-j} $, $\tilde K^{i,j} \in \mathbb{R}^{n_\mathrm{u} \times n_\mathrm{x}}$. Using this feedback, we can write the closed-loop error dynamics as
\begin{equation}
\tilde \X = \mathcal{Z}(\tilde{\mathcal{A}} + \tilde{\mathcal{B}} \tilde {\mathcal{K}})\tilde \X+\tilde\W,~\tilde \U = \tilde {\mathcal{K}}\tilde \X,~ \tilde x_0=0_{n_\mathrm{x}},
\label{eq:closed-loop}
\end{equation}
or equivalently as
\begin{equation}
\begin{bmatrix}\tilde \X\\ \tilde \U \end{bmatrix} = \begin{bmatrix} (\mathcal{I} - \mathcal{Z}\tilde{\mathcal{A}} - \mathcal{Z}\tilde{\mathcal{B}}\tilde {\mathcal{K}} )^{-1} \\ \tilde{\mathcal{K}}(\mathcal{I} - \mathcal{Z}\tilde{\mathcal{A}} - \mathcal{Z}\tilde{\mathcal{B}}\tilde{\mathcal{K}})^{-1} \end{bmatrix} \tilde\W = \vcentcolon \begin{bmatrix} \tilde{\Phi}_\mathrm{x} \\ \tilde{\Phi}_\mathrm{u} \end{bmatrix}\tilde \W.\\
\label{eq:sys_rep}
\end{equation}
In the following, we directly optimize the matrices
\begin{equation}
\begin{aligned}
    \tilde{\Phi}_\mathrm{x} &= \begin{bmatrix}
      \tilde{\Phi}_\mathrm{x}^{0,0} &&\\
      \vdots &\ddots&\\
      \tilde{\Phi}_\mathrm{x}^{T-1,T-1} & \cdots & \tilde{\Phi}_\mathrm{x}^{T-1,0}\\
    \end{bmatrix} \in \mathcal{L}^{T,n_\mathrm{x} \times n_\mathrm{x}},
\end{aligned}
\end{equation}
and $\tilde{\Phi}_\mathrm{u} \in \mathcal{L}^{T,n_\mathrm{u} \times n_\mathrm{x}}$, called the system responses from the disturbance to the closed-loop state and input, respectively.
The following proposition states that the closed-loop response under arbitrary linear feedback is given by all system responses in a linear subspace.
\begin{proposition}\cite[adapted from Theorem 2.1]{Anderson2019}
\label{prop:1}
Let $\tilde \W\in \tilde{\mathbb{W}}^{0:T-1}$ be an arbitrary disturbance sequence. Any $\tilde \X$, $\tilde \U$ satisfying~\eqref{eq:closed-loop}, also satisfy~\eqref{eq:sys_rep} with some $\tilde{\Phi}_x \in \mathcal{L}^{T,n_\mathrm{x} \times n_\mathrm{x}}$, $\tilde{\Phi}_u \in \mathcal{L}^{T,n_\mathrm{u} \times n_\mathrm{x}}$ lying on the affine subspace
\begin{equation}
\left.
    \begin{aligned}
        &\left[ \mathcal{I} - \mathcal{Z}\tilde{\mathcal{A}}\quad -\mathcal{Z}\tilde{\mathcal{B}}\right] \left[\begin{array}{c}\tilde{\Phi}_\mathrm{x}\\\tilde{\Phi}_\mathrm{u} \end{array}\right] = \mathcal{I}.
    \end{aligned}
    \right.    
    \label{eq:affine_map}
\end{equation}
Let $\tilde{\Phi}_\mathrm{x}$ and $\tilde{\Phi}_\mathrm{u}$ be arbitrary matrices satisfying~\eqref{eq:affine_map}. Then the corresponding $\tilde \X$ and $\tilde \U$ computed with~\eqref{eq:sys_rep} also satisfy~\eqref{eq:closed-loop} with $\tilde{\mathcal{K}} = \tilde{\Phi}_\mathrm{u}  \tilde{\Phi}_\mathrm{x}^{-1}\in \mathcal{L}^{T,n_\mathrm{u}\times n_\mathrm{x}}$.

\end{proposition}
\begin{proof}
    The proof follows directly from~\cite[Theorem 2.1]{Anderson2019} for systems with zero initial conditions.
\end{proof}
\begin{remark}
We note that the inverse of $\tilde \Phi_\mathrm{x}$ never has to be computed explicitly, not even for the implementation of the error feedback $\tilde{\mathcal{K}}$~\cite{Anderson2019}.
\end{remark}
Proposition~\ref{prop:1} allows to compute the gains $\tilde{\mathcal{K}}$ using linear parametrization, for any \ac{LTV} system and in particular for~\eqref{eq:LTV_based_system}, where the matrices ${\mathcal{A}}$ and ${\mathcal{B}}$ depend on the nominal trajectory $(z,v)$.
Next, we show how the parameterization~\eqref{eq:affine_map} can be utilized to exactly solve Problem~\eqref{eq:prob_form} with linear feedback $\mathcal{K}$ for a given nominal trajectory $(z,v)$ and LTV error dynamics~\eqref{eq:sys_dyn}. To this end, we consider a causal linear error feedback of the form
\begin{equation}
    \tilde \U = \tilde{\mathcal{K}}\tilde \X, ~ \tilde u_0 =\tilde  v_0,~\tilde{\mathcal{K}}\in\mathcal{L}^{T,n_{\mathrm{u}}\times n_{\mathrm{x}}},
    \label{eq:affine_contr}
\end{equation}
and apply Proposition~\ref{prop:1} to characterize the system response of the \ac{LTV} error system~\eqref{eq:sys_dyn}.
The closed-loop error on the states and inputs is expressed using the definition of $\tilde{\Phi}_\mathrm{x}$ and $\tilde{\Phi}_\mathrm{u}$ in~\eqref{eq:sys_rep} for the \ac{LTV} error system~\eqref{eq:sys_dyn}, i.e., 
\begin{equation}
    \sv{\tilde x_k}{\tilde u_k} =  \sum_{j=0}^{k-1} \tilde{\Phi} ^{k-1,j}\tilde w_{k-1-j}\K,
    \label{eq:decomp}
\end{equation}
where $\tilde \Phi^{k,j} \defmath  \sv{\tilde{\Phi}_\mathrm{x}^{k,j}}{\tilde{\Phi}_\mathrm{u}^{k,j}}$, $\tilde{\Phi}_\mathrm{x}\in \mathcal{L}^{T,n_\mathrm{x}\times n_\mathrm{x}}$ and $\tilde{\Phi}_\mathrm{u}\in \mathcal{L}^{T,n_\mathrm{u}\times n_\mathrm{x}}$. Given Proposition~\ref{prop:1}, we can provide tight, i.e., necessary and sufficient, conditions for robust state and input constraint satisfaction for the \ac{LTV} system.
\begin{proposition}
\label{prop:lin}

There exists a causal error feedback of the form~\eqref{eq:affine_contr} such that for any $\tilde{\W}\in \tilde{\mathbb{W}}^T$
\begin{equation}
    c_i\T\sv{x_k}{  u_k}  + b_i \le 0\K\I,
 \label{eq:lin_const}
\end{equation}
with $\sv{ x_k}{ u_k} = \sv{ \tilde x_k}{ \tilde u_k}+ \sv{ z_k}{  v_k}$ according to~\eqref{eq:sys_dyn}, if and only if there exist matrices $\tilde{\Phi}_\mathrm{x}$, $\tilde{\Phi}_\mathrm{u}$ satisfying~\eqref{eq:affine_map}, and$\K\I$,
\begin{equation}
\begin{aligned}
&c_i\T \sv{\tilde z_k}{\tilde v_k}  + b_i+\sum_{j=0}^{k-1}  \| c_i\T \tilde{\Phi} ^{k-1,j} \tilde E_{k-1-j} \|_1 \le 0 .
    \end{aligned}
    \label{eq:nom_noise}
\end{equation}
\end{proposition}
\begin{proof}
As per Proposition~\ref{prop:1}, the \ac{LTV} error system can be written with Equation~\eqref{eq:decomp}, and we can apply directly~\cite[Example 8]{Goulart2006OptimizationConstraints}. Namely,$\K \I$, we have
\begin{IEEEeqnarray}{rCll}
    &&\max_{\tilde{\W}_{0:k-1} \in \tilde{\mathbb{W}}^{0:k-1}} c_i\T \sv{  x_k}{  u_k} +b_i \\
    &&\stackrel{\eqref{eq:decomp}}{=}  c_i\T \sv{ z_k}{ v_k} + \max_{\tilde{\W}_{0:k-1} \in \tilde{\mathbb{W}}^{0:k-1}} \sum_{j=0}^{k-1} c_i\T\tilde{\Phi} ^{k-1,j} \tilde w_{k-1-j} +b_i\nonumber\\
    &&= c_i\T \sv{ z_k}{ v_k} + \sum_{j=0}^{k-1} \max_{\tilde w_{j} \in \tilde{\mathbb{W}}_{j}}  c_i\T\tilde{\Phi} ^{k-1,j} \tilde w_{k-1-j} +b_i\nonumber\\    &&\stackrel{\eqref{eq:noise_set_lin}}{=} c_i\T \sv{ z_k}{ v_k} + \sum_{j=0}^{k-1} \left\|c_i\T\tilde{\Phi} ^{k-1,j}\tilde E_{k-1-j}\right\|_1 +b_i.\nonumber& \qedhere
\end{IEEEeqnarray}
\end{proof}
\begin{remark}
A similar reformulation for ellipsoidal disturbances or more general polytopic disturbances can be found in~\cite[Example 7-8]{Goulart2006OptimizationConstraints}. Likewise, the result can be naturally extended to time-varying constraints.
\end{remark}
We have presented conditions on a linear error feedback to achieve guaranteed constraint satisfaction assuming the error dynamics are given by an \ac{LTV} system using the linear error feedback parameterization~\eqref{eq:affine_map}, a description of the disturbance set $\tilde{\mathbb{W}}$~\eqref{eq:noise_set_lin}, and a nominal trajectory $( z,  v)$. 
In combination, Proposition~\ref{prop:1} characterizes the system response of the error dynamics~\eqref{eq:decomp}, while Proposition~\ref{prop:lin} provides tight bounds on the nominal trajectory and system response to ensure robust constraint satisfaction.
By combining these two results, we can solve the robust optimal control problem~\eqref{eq:prob_form} for \ac{LTV} error dynamics~\eqref{eq:sys_dyn}, error feedback~\eqref{eq:affine_contr}, and disturbances of the form~\eqref{eq:noise_set_lin} using the following optimization problem:
\begin{IEEEeqnarray}{rCl}
\IEEEyesnumber\label{eq:sls_lin}
\IEEEyessubnumber*
    \min_{ \tilde{\Phi}} \quad &&  J_T(\tilde{\Phi}),\\
    \text{s.t.}\quad  && \left[ \mathcal{I} - \mathcal{Z}\tilde{\mathcal{A}}\quad -\mathcal{Z}\tilde{\mathcal{B}}\right] \tilde{\Phi}  = \mathcal{I},\label{eq:sls_lin_slp}\\
    && \sum_{j=0}^{k-1}  \| c_i\T \tilde{\Phi} ^{k-1,j} \tilde E_{k-1-j} \|_1 \\
    &&~+c_i\T \sv{ z_k}{ v_k} + b_i  \le 0 \K\I,\nonumber
\end{IEEEeqnarray}
where $\tilde{\Phi}\defmath (\tilde{\Phi}_\mathrm{x},\tilde{\Phi}_\mathrm{u})$. Specifically, the solution of~\eqref{eq:sls_lin} provides an optimized error feedback $\tilde{\mathcal{K}} = \tilde{\Phi}_\mathrm{u} \tilde{\Phi}_\mathrm{x}^{-1}$ which guarantees robust satisfaction of the constraints for the closed-loop state and input trajectories. Problem~\eqref{eq:sls_lin} is a reformulation of established results in the literature (see, e.g.,~\cite{Anderson2019, Sieber2021AControl,wang2019system}), optimizing a linear error feedback controller for tight robust constraint satisfaction. 

In the following, we extend the linear formulation~\eqref{eq:sls_lin} to 
address the nonlinear robust optimal control problem~\eqref{eq:prob_form}: we jointly optimize a nominal nonlinear trajectory~\eqref{eq:nominal_trajectory} and utilize the \ac{LTV} dynamics~\eqref{eq:LTV_based_system} parameterized by the nominal nonlinear dynamics.
Besides, we account for the unknown error $\|e\|_\infty$ in~\eqref{eq:eta_kdef}, leveraging \ac{SLS} to jointly optimize a dynamic over-bound. In particular, we do not rely on an offline over-bound of $\|e\|_\infty$, as it would lead to an overly conservative formulation, as showcased in Section~\ref{subsubs:linear_comp}.
\subsection{Robust nonlinear finite-horizon optimal control problem}
\label{sec:result}
In this section, given the parameterization of the affine error feedback for the \ac{LTV} system (Section~\ref{sec:SLS}) and the linearization error of the nonlinear system (Section~\ref{sec:linerror}), we are in a position to introduce the main result of this paper (cf. Fig.~\ref{fig:outline}).  
In particular, we will show in Theorem~\ref{thm1} that the following \ac{NLP} provides a feasible solution to the robust optimal control problem in~\eqref{eq:prob_form}:
\begin{IEEEeqnarray}{rCl}
\IEEEyesnumber\label{eq:full_NLP}
\IEEEyessubnumber*
    \min_{\Z, v_0,\V,{\Phi} ,\D} &&J_T(\xz, \Z,\V,\Phi),\label{eq:NLP}\\
    \text{s.t. }  &&     \left[ \mathcal{I} - \mathcal{Z}\mathcal{A}(\Z,\V)\quad - \mathcal{Z}\mathcal{B}(\Z,\V)  \right]
    \Phi= \mathcal{I},\label{eq:nlp1}\\
    && z_{k+1} = f(z_k, v_k)\Km,z_0 = \xz,\label{eq:nlp2}\\
        &&\sum_{j=0}^{k-1}\| c_i\T \Phi ^{k-1,j} [E,\tube_{k-1-j}^2 \mu ] \|_1\label{eq:rob_lin_constraints_a} \\
    && ~+  c_i \T \zv + b_i \le 0\K\I,\nonumber\\
    && \sum_{j=0}^{k-1}\| \Phi ^{k-1,j}[E,\tube_{k-1-j}^2 \mu ]\|_{\infty} \le \tube_k\Km,\label{eq:rob_lin_constraints_b}
\end{IEEEeqnarray}
where we denote a feasible solution as $\{\Z^\feas, v_0^\feas, \V^\feas, \Phi^\feas,\bm{\tube}^\feas\}$, with $\Z \defmath  \left( z_1, \dots, z_T\right)\in \mathbb{R}^{Tn_\mathrm{x}}$, $\V \defmath  \left( v_1,\dots, v_T \right)\in \mathbb{R}^{Tn_\mathrm{u}}$, $\mathcal{A}(\Z,\V) \defmath\text{blkdiag}( A_1, \dots,  A_{T-1},0_{n_\mathrm{x},n_\mathrm{x}})\in\mathcal{L}^{T,n_\mathrm{x}\times n_\mathrm{x}}$, $\mathcal{B}(\Z,\V)\defmath  \text{blkdiag}( B_1, \dots,  B_{T-1},0_{n_\mathrm{x},n_\mathrm{u}})\in\mathcal{L}^{T,n_\mathrm{x}\times n_\mathrm{u}}$, $\D  = (\tube_0,\dots, \tube_{T-1})\in \mathbb{R}^{T}$,  ${\Phi}_\mathrm{x}\in \mathcal{L}^{T,n_\mathrm{x} \times n_\mathrm{x}}$, ${\Phi}_\mathrm{u}\in \mathcal{L}^{T,n_\mathrm{u} \times n_\mathrm{x}}$ and $\mu$ according to~\eqref{eq:def_mu}.
The nominal prediction is given by~\eqref{eq:nlp2}. Equation~\eqref{eq:nlp1} computes the system response for the linearization around the nominal trajectory $\Z$, $\V$ (cf. Proposition~\ref{prop:1}). The auxiliary variable $\tube_k$ is introduced to upper bound $\|e_k\|_\infty$, which is used to obtain a bound on the linearization error (cf. Proposition~\ref{prop:1}), that depends on all previous $\tube_j$, $j=0, \dots, k-1$, giving~\eqref{eq:rob_lin_constraints_b}. Using Proposition~\ref{prop:lin}, the constraints are tightened with respect to both the additive disturbance $w_k\in \mathbb{W}$ and the linearization error $ \|e_k\|_\infty \mu$ combined as in~\eqref{eq:eta_kdef}, i.e., the additive uncertainties on the \ac{LTV} error lie in the set $E\B^{n_\mathrm{w}} \oplus \|e \|_\infty^2\mu \B^{n_\mathrm{x}}$. As a result, the reachable set of the nonlinear system~\eqref{eq:nonlinear_dyn} at time step $k$, in closed-loop with the affine error feedback computed as in Theorem~\ref{thm1} satisfies
\begin{equation}
x_k\in z_k^\feas\bigoplus_{j=0}^{k-1} \Phi_\mathrm{x}^{\feas k-1,j}[E, {\tube^{\feas 2}_{k-1-j}}\mu] \B^{n_\mathrm{x} + n_\mathrm{w}}\mathdef \mathbb{D}_k \K.
\label{eq:tubes}
\end{equation}

The following theorem summarizes the properties of the proposed NLP~\eqref{eq:full_NLP}.
\begin{theorem}
\label{thm1}
Given Assumptions~\ref{assum:1},~\ref{assum:2} and~\ref{assum:3}, suppose the optimization problem~\eqref{eq:full_NLP} is feasible.
Then, the affine error feedback $\U = \V^\feas + \mathcal{K}^\feas(\X-\Z^\feas)$, $\mathcal{K}^\feas = \Phi_\mathrm{u}^\feas{\Phi_\mathrm{x}^\feas}^{-1}$, $u_0 = v_0^\feas$ provides a feasible solution to Problem~\eqref{eq:prob_form}, i.e.,
the closed-loop trajectories of system~\eqref{eq:nonlinear_dyn} under this error feedback robustly satisfy the constraints~\eqref{eq:prob_form_cons}.
\end{theorem}
\begin{proof}
First, the constraints~\eqref{eq:nlp2} ensure that the nominal trajectory satisfies the dynamics~\eqref{eq:nominal_trajectory}. 
Then, we use a Taylor series approximation with respect to the nominal trajectory~\eqref{eq:nlp2}, resulting in the \ac{LTV} error system~\eqref{eq:LTV_based_system}.
We apply Proposition~\ref{prop:1} to the \ac{LTV} error system~\eqref{eq:LTV_based_system} and the constraint~\eqref{eq:nlp1} implies that the closed-loop trajectories of the error system satisfy
\begin{equation}
    \begin{bmatrix}
    \X - \Z\\
    \U - \V
    \end{bmatrix} = 
    \begin{bmatrix}
    \Phi_\mathrm{x}\\\Phi_\mathrm{u}
    \end{bmatrix}
    \E,
\label{eq:cl_nl}
\end{equation}
with $\E \defmath (\d_0,\dots,\d_{T-1})\in \mathbb{R}^{Tn_\mathrm{x}}$ given by~\eqref{eq:eta_kdef}, $\X$, $\U$ satisfying~\eqref{eq:nonlinear_dyn} and $\Z$, $\V$ satisfying~\eqref{eq:nominal_trajectory}.
In the following, we show by induction that the auxiliary variables $\D\in \mathbb{R}^T$ satisfy 
\begin{equation}
    \| e_j \|_\infty \le \tube_j~\forall j\in \mathbb{N}_{T} .\label{eq:inclu}
\end{equation}
Inequality~\eqref{eq:inclu} holds for $j=0$ since $\|e_0\|_\infty=0$ (cf.~\eqref{eq:rob_lin_constraints_a}) and $ \tube_0\geq 0$ (cf.~\eqref{eq:rob_lin_constraints_b}).
Note that, as per Proposition~\ref{prop:bound_remainder}, the disturbance on the \ac{LTV} error system $\d_k$ satisfies~\eqref{eq:eta_kdef}.
Then, assuming Inequality~\eqref{eq:inclu} holds $ \forall j\in \mathbb{N}_{k-1} $, we have
\begin{equation}
    \d_j\in [E,\|e_j\|_\infty^2 \mu]\B^{n_\mathrm{x}+n_\mathrm{w}} \subseteq [E,\tube_j^2 \mu]\B^{n_\mathrm{x}+n_\mathrm{w}}~\forall j\in \mathbb{N}_{k-1} .
    \label{eq:new_E}
 \end{equation}
Hence, we obtain
\begin{IEEEeqnarray}{rll}
 \| \ek\|_\infty &\stackrel{\eqref{eq:cl_nl}}{=} \left\|\sum_{j=0}^{k-1}\Phi ^{k-1,j} \d_{k-1-j} \right\|_{\infty}&\label{eq:nonlin_noise}\\
 &\stackrel{\eqref{eq:new_E}}{\le} \sum_{j=0}^{k-1} \|  \Phi ^{k-1,j} [E,\tube_{k-1-j}^2 \mu ] \|_{\infty} &\stackrel{\eqref{eq:rob_lin_constraints_b}}{\le} \tube_k.\nonumber
\end{IEEEeqnarray}
Therefore, the constraints~\eqref{eq:rob_lin_constraints_b} ensure that Inequality~\eqref{eq:inclu} holds for any realization of the disturbance.
Finally, the constraint~\eqref{eq:rob_lin_constraints_a} in combination with Equation~\eqref{eq:new_E} ensures that the constraints are robustly satisfied, analogously to Proposition~\ref{prop:lin}, which can be seen by substituting $\tilde E_k$ for $[E,\tube_k^2 \mu]$ with $n_\mathrm{\tilde w} = n_\mathrm{x}+n_\mathrm{w}$.
\end{proof}
\begin{remark}
The optimization of the error (bound) dynamics, error feedback controller and nominal nonlinear trajectory are intricately intertwined in the optimization problem~\eqref{eq:full_NLP}.
In particular, we utilize the error dynamic
\begin{equation}
\label{eq:error_bounds_dynamics}
    e_k \in \bigoplus_{j=0}^{k-1} \Phi^{ k-1,j}[E, \|e_{k-1-j}\|_\infty^{ 2}\mu] \B^{n_\mathrm{x} + n_\mathrm{w}},
\end{equation}
to derive the constraint~\eqref{eq:rob_lin_constraints_b} which yields the jointly-optimized dynamical linearization error over-bound.
Although formulated differently from homothetic tube \ac{MPC}~\cite{Kohler2021ASystems, Rakovic2023}, this formulation also enables the error bound $\tau_k$ to grow and shrink.
\end{remark}

\begin{remark}
When the disturbance $w$ is set to zero with $E = 0_{n_\mathrm{x}, n_{{\mathrm{w}}}}$, we recover a nominal (non-robust) trajectory optimization problem (cf., e.g.,~\cite{MALYUTA2021282}), regardless of the value of worst-case curvature $\mu$, as the constraints~\eqref{eq:rob_lin_constraints_a} reduce to a nominal constraint, with $\D=0_T$, independent of $\Phi$. 
In case the system is linear ($\mu = 0_{n_{\mathrm{x}},n_{\mathrm{x}}}$), we recover the linear \ac{SLS} formulation (cf.~\cite{Sieber2021AControl}) since $\D$ does not enter the constraints~\eqref{eq:rob_lin_constraints_a}. In that case, ~\eqref{eq:error_bounds_dynamics} yields the exact reachable set under affine feedback.
\end{remark}

\begin{remark}
The considered handling of the nonlinear system is based on a linearization along an online optimized nominal trajectory and is comparable to~\cite{Villanueva2017RobustInequalities, Althoff2008ReachabilityLinearization, Messerer2021AnFeedback, Kim2022JointNonlinearities}, where the latter results also provide an affine feedback. 
However, the over-approximation of the reachable set in~\cite{Messerer2021AnFeedback} is based on the ellipsoidal propagation in~\cite{Houska2011RobustSystems}, where even the linear case is not tight. In contrast to~\cite{Cannon2011RobustControl}, both the linearized system $\mathcal{A}$, $\mathcal{B}$ and the bound on the remainder term $\D$ are adjusted online based on the jointly optimized nominal trajectory $\Z$, $\V$, which avoids conservativeness.
\end{remark}

\FloatBarrier
\section{Case Study: Satellite Attitude Control}
\label{sec:case_study}
The following example demonstrates the benefits of the proposed approach where we optimize jointly the error feedback and the nominal trajectory for nonlinear systems, and the dynamic linearization error overbound with additive disturbances. 
We apply the proposed method for the constrained attitude control of a satellite~\cite{MALYUTA2021282}, using an efficient \ac{SQP} implementation based on inexact Jacobians\footnote{An open-source implementation is available at \url{https://gitlab.ethz.ch/ics/nonlinear-system-level-synthesis}, doi: \url{https://doi.org/10.3929/ethz-b-000682052}}.
\begin{figure*}
    \centering
    \includegraphics[width=\textwidth]{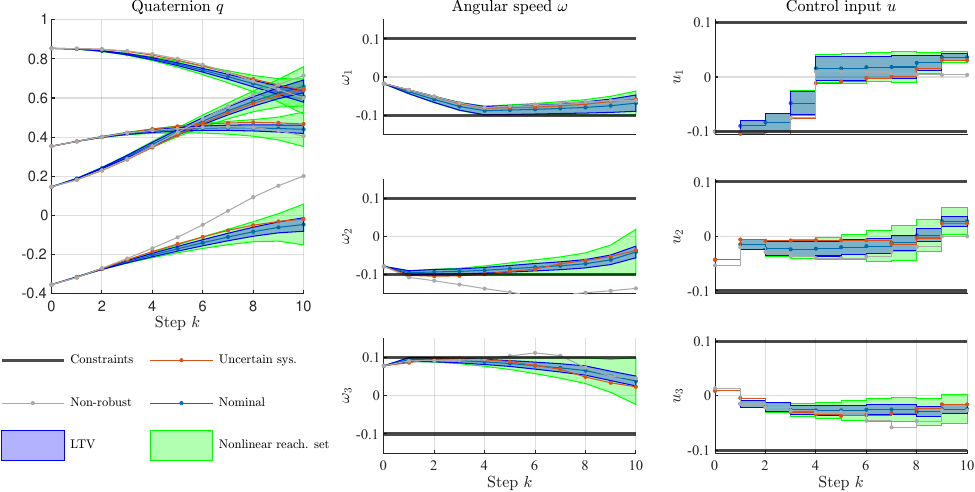}
    \caption{Robust nonlinear optimal control for the attitude control of a spacecraft computed with inexact \ac{SQP}. For each state and input, the reachable set (see Equation~\eqref{eq:tubes}) around the nominal trajectory depicts the online-computed reachable sets (green) and its \ac{LTV} approximation (blue), i.e., $\mu =0$. The reachable sets are designed to remain within the constraints (black), and the sample trajectory (red) is hence guaranteed to stay therein, unlike the nominal trajectory optimization design (grey).}
    \label{fig:att_spacecraft}
\end{figure*}

\subsection{System and constraints}
We consider the following Euler's equation of the dynamics
\begin{equation}
    \dot z = \begin{bmatrix}
    \Omega(\omega)q\\
    I_S^{-1}(v-\omega \times (I_S \omega))
    \end{bmatrix},
    \label{eq:nonlin_dyn_rot}
\end{equation}with states $z \defmath (q, \omega)$, attitude quaternion $q\in \mathbb{R}^4$, angular rotation rate $\omega\in \mathbb{R}^3$, input torque $v\in \mathbb{R}^{n_\mathrm{u}}$, $n_\mathrm{x} = 7$,  $n_\mathrm{u} =3$ and $\times$ is the cross product. The symmetric inertia matrix of the satellite is $I_S = \text{diag}(5,2,1)\in \mathbb{R}^{3\times 3}$ and
\begin{equation}
    \Omega(\omega) \defmath \frac{1}{2}\begin{bmatrix}
     0& -\omega_1 & -\omega_2 & -\omega_3\\
    \omega_1& 0 & \omega_3 & -\omega_2\\
        \omega_2& -\omega_3 & 0 & \omega_1\\
        \omega_3& \omega_2 & -\omega_1 & 0\\
    \end{bmatrix},
\end{equation}
with $\Omega : \mathbb{R}^3\mapsto \mathbb{R}^{4\times 4}$.
The dynamics are discretized using the $4^\text{th}$ order Runge–Kutta method, using a time step of one, 
which results in a negligible discretization error.
We consider a bounded disturbance applied to the system described by~\eqref{eq:noise_set} with $E =  5 \cdot 10^{-3} [ 0_{3, 4}~ \mathcal{I}_3]^\top \in \mathbb{R}^{7\times n_\mathrm{w}}$ with $n_\mathrm{w}=3$, which could stem, e.g., from the solar radiation pressure, the flexible modes of the solar panels, the aerodynamic drag, or any unmodelled dynamics.
Additionally, we consider the constraints $-0.1\le\omega_i\le 0.1$ and $-0.1 \le v_i\le 0.1$ for $i=1,2,3$. We use the nominal cost function
 $J_T(\xz, \Z,\V)=\sum_{k=0}^{T-1}\ell(z_k,v_k)+\ell_T(z_T)$, with
 stage cost $\ell(z,v) = (z - z_\text{ref})\T {Q}(z - z_\text{ref}) + (v - v_\text{ref})\T {R}(v - v_\text{ref})$,
and the terminal cost $\ell_T(z) = (z-z_\text{ref}) \T {Q} (z-z_\text{ref})$, with ${Q} = 0.7 \mathcal{I} \in \mathbb{R}^{7\times 7} $, ${R} = \mathcal{I}\in \mathbb{R}^{3 \times 3}$, the reference $z_\text{ref} = (1,0,0,0,0,0,0)\T$, $ v_\text{ref} = (0,0,0)\T$, and the horizon is $T = 10$.
As often seen in numerical optimization, we consider an additional term to the cost function, i.e., we use $J_T + \alpha \Y\T\Y$, with $\alpha = 10^{-2}$.
For simplicity, we approximate the constant $\mu \approx \text{diag}(
3.699  ,  3.703   , 3.717  ,  3.635 ,   0.649 ,   4.608   , 5.635) \in \mathbb{R}^{7 \times 7}$ from Equation~\eqref{eq:def_mu} for the nonlinear dynamics~\eqref{eq:nonlin_dyn_rot} by using a Monte-Carlo method with $10^4$ samples for a total offline computation time of 444 seconds.
The initial condition is $\xz = (\bar q, \bar \omega)$, where $\bar q$ is inferred from the Euler angles $(180, 45, 45)\cdot \frac{\pi}{180}$ and $\bar \omega = (-1, -4.5, 4.5)\cdot \frac{\pi}{180}$.

\subsection{Inexact \ac{SQP} implementation}
\label{sec:sqp}
To compute a solution to the optimization problem~\eqref{eq:full_NLP}, we employ an \ac{SQP} implementation, formulated with CasADi~\cite{Andersson2019}, which iteratively solves a \ac{QP} with Gurobi~\cite{gurobi}, locally approximating the \ac{NLP} at the previous iteration. For an efficient implementation, we neglect the Jacobians of the constraint~\eqref{eq:nlp1} with respect to $(\Z,\V)$, i.e., the employed inexact \ac{SQP} uses the following inexact linearization around $(\hat \Z, \hat \V)$,
\begin{IEEEeqnarray}{rCl}
            \left[ \mathcal{I} - \mathcal{Z}\mathcal{A}(\hat \Z, \hat \V)\quad \right. &&\left.- \mathcal{Z}\mathcal{B}(\hat \Z, \hat \V)  \right] \Phi= \mathcal{I}.
            \label{eq:NLP_trun}
\end{IEEEeqnarray}

As SQP approaches require twice differentiable constraints~\cite{Bock2007NumericalControl}, and the constraint \eqref{eq:NLP_trun}, or \eqref{eq:nlp1} contain the dynamics' Jacobian, this requires three times differentiable dynamics.
In addition, a Gauss-Newton approximation $H_J$ of the cost is used, 
with a small regularization term, i.e., we use $H_J + \gamma \mathcal{I}_{n_\mathrm{y}}$, with $\gamma=10^{-2}$. We assume convergence when the condition $\|(\Delta \Y,\Delta \bm{\nu})\|_\infty \le 10^{-6}$ is fulfilled, where $\Delta \bm{y}$ and $\Delta \bm{\nu}$ are respectively the primal and dual step between consecutive \ac{QP} steps~\cite{WrightStephenandNocedal1999NumericalOptimization}.A tailored solver was recently developed, exploiting the numerical structure of the QPs\cite{leeman2024fast}.

Under several mild assumptions, especially on the norm of the Jacobian being neglected, inexact \ac{SQP} converges to a feasible, but suboptimal, solution of the original \ac{NLP}~\eqref{eq:full_NLP}, while a standard implementation would converge to a locally optimal solution (see, e.g.,~\cite{Bock2007NumericalControl}) 
\subsection{Results}
\begin{figure}[ht!]
    \centering
    \includegraphics[width = 0.45\textwidth]{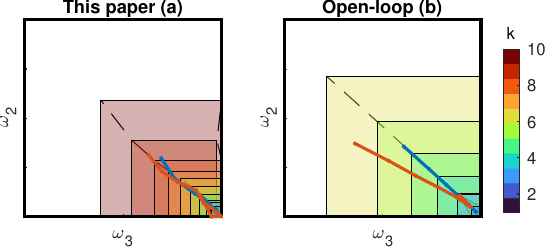}
    \caption{Comparison between the reachable set of the proposed robust nonlinear optimal control~\eqref{eq:full_NLP}~(a), and the (reduced-horizon) open-loop robust nonlinear optimal control (b), shown with an exemplary trajectory subject to disturbances (red).}
    \label{fig:pitch5}
\end{figure}

\subsubsection{Nonlinear system with affine error feedback}

For the problem considered, Fig.~\ref{fig:att_spacecraft} shows the solution of the NLP~\eqref{eq:full_NLP} computed with inexact \ac{SQP}, which ensures that the trajectories of the nonlinear system remain robustly within the constraints.
The shaded areas correspond to the reachable sets of the nonlinear system (Equation~\eqref{eq:tubes}) (green) and its \ac{LTV} approximation (blue), i.e., $\mu =0$. An illustrative disturbance sequence is applied and, as ensured by the proposed design, the resulting trajectory (red) remains within the reachable sets. The flexible error feedback parameterization allows the tubes to grow in some directions and shrink in others to meet the constraints, which illustrates the flexibility of this method. 
Moreover, Fig.~\ref{fig:att_spacecraft} compares our method with its nominal non-robust counterpart, which does not optimize error feedback ($\bm{\Phi}_\mathrm{x} = 0$, $\bm{\Phi}_\mathrm{u} = 0$) and neglects disturbances ($E = 0$).
The corresponding nominal non-robust (grey) trajectory with disturbances violates the constraints.
The problem was solved on an i9-7940X processor with 32GB of RAM memory, in 18 QP iterations, for a total of 4.45 seconds.
\subsubsection{Comparison with open-loop approach}
To highlight the benefits of the proposed method, we solve the NLP~\eqref{eq:full_NLP}, with $\Phi_\mathrm{u}=0_{T n_\mathrm{u},Tn_\mathrm{x}}$, resulting in an open-loop robust formulation, i.e., $\mathcal{K}=0_{T n_\mathrm{u},Tn_\mathrm{x}}$. For the open-loop case only, we consider a reduced horizon of $T=6$, as a longer horizon leads to infeasibilities. 
Indeed, this open-loop robust method does not apply error feedback, which makes the reachable set effectively larger and limits both the size of the disturbance and the horizon that can be considered. 
The open-loop robust formulation maintains the guarantees of robust constraint satisfaction, as the reachable set or tube always stays within the constraints as depicted in Fig.~\ref{fig:pitch5}(b). 
Because of the large size of the reachable set, the nominal trajectory is forced to move away from the constraints, leading to poor performance.
Therefore, the affine error feedback is instrumental to ensure robust constraint satisfaction for large disturbances without acting overly conservatively.

\subsubsection{Comparison with linear approaches}
\label{subsubs:linear_comp}
\begin{figure}
    \centering
    \includegraphics[width = 0.45\textwidth]{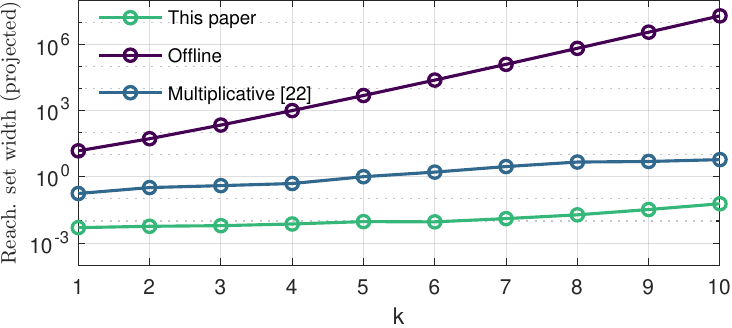}
    \caption{
    Width of the reachable set projected on $\omega_2$. Comparison between the proposed method, a linear \ac{SLS} method with offline-overbound, and a linear \ac{SLS} with multiplicative uncertainties \cite{Chen2021SystemApproximation}.}
    \label{fig:tubes_comp}
\end{figure}
We compare the performance of our method with existing linear \ac{SLS} alternatives, i.e., with similar joint trajectory and error feedback optimization, but linear nominal system and overbounding of the nonlinearity. First, for the sake of comparison with~\cite{Sieber2021AControl,wang2019system}, we consider the dynamics linearized at $(z_\text{ref}, v_\text{ref})$ and compute a polytopic disturbance bound accounting for the maximal linearization error w.r.t. $f(x,u)$ for all $(x,u)\in \mathbb{P}$.
Then, we also consider a multiplicative uncertainty formulation~\cite{Chen2021SystemApproximation} which requires $2^{n_\mathrm{x}} = 128$ 
inequality constraints for the overbound at each time step, {as the dynamics are bilinear} and we utilized a forward-Euler integrator.
Both linear methods yield an un-controllable linear system, and hence unbounded tube growth, which results in infeasible problems.

For comparison, Fig.~\ref{fig:tubes_comp} shows the width of the reachable set for the state $\omega_2$ for our method and both linear techniques, all three using the same feedback $\bm{\Phi}_\textrm{u}^\feas$ and nominal input $\V^\feas$. 
For this comparison, we utilized a forward-Euler integrator for simplicity.
The proposed formulation leads to order of magnitudes conservativeness reduction.

\section{Conclusion}
We proposed a novel approach to solve finite horizon constrained robust optimal control problems for nonlinear systems subject to additive disturbances, as e.g. commonly encountered in trajectory optimization or MPC. One of the main novelties lies in the joint optimization of a nonlinear nominal trajectory, an affine error feedback and dynamic over-bounds of the linearization error, ensuring robust constraint satisfaction. 
We showcased the method for the control of a rigid body in rotation with constraints on the states and inputs, and demonstrated reduced conservatism of our method compared to (nonlinear) open-loop robust and existing linear approaches jointly optimizing the nominal trajectory and error feedback. Considering a receding horizon implementation, as is typical in robust MPC, including a corresponding recursive feasibility and stability analysis, is an important direction for future work.
We recently extended the proposed method to handle parametric uncertainties\cite{leeman2023robust}.
\FloatBarrier
\bibliography{references}

\end{document}